\documentclass{opt2020} 

\usepackage{amsmath, amsfonts, amssymb, algorithm, graphicx, mathtools,xfrac, tikz-cd, xcolor, float}
\usepackage[utf8]{inputenc} 
\usepackage[T1]{fontenc}    
\usepackage{hyperref}       
\usepackage{url}            
\usepackage{booktabs}       
\usepackage{amsfonts}       
\usepackage{nicefrac}       
\usepackage{microtype}      

\newcommand{\calP}{\mathcal{P}}

\renewcommand{\P}{\mathcal{P}}

\newcommand{\Mr}{\mathcal{M}_{\textbf{r}}}

\newcommand{\rank}{\mathrm{rank}}

\newcommand{\D}{\mathrm{D}}
\newcommand{\De}{\emph{D}}
\newcommand{\Xt}{\widetilde{X}}
\newcommand{\Ut}{\widetilde{U}}

\newcommand{\reals}{\mathbb{R}}
\newcommand{\R}{\mathrm{R}}
\newcommand{\inner}[2]{\left\langle{#1},{#2}\right\rangle}

\newcommand{\T}{\mathrm{T}}
\newcommand{\grad}{\mathrm{grad}}
\newcommand{\Hess}{\mathrm{Hess}}

\title[Second-order optimization for tensors with fixed tensor-train rank]{Second-order optimization for tensors with fixed tensor-train rank}

\optauthor{%
\Name{Michael Psenka} \Email{mpsenka@princeton.edu}\\
\addr Princeton University, Princeton, NJ, USA  
\AND
\Name{Nicolas Boumal} \Email{nicolas.boumal@epfl.ch}\\
\addr EPFL, Switzerland%
}

\makeatletter
 \let\Ginclude@graphics\@org@Ginclude@graphics 
\makeatother
\begin{document}

\maketitle

\begin{abstract}%
There are several different notions of ``low rank'' for tensors, associated to different formats. Among them, the Tensor Train (TT) format is particularly well suited for tensors of high order, as it circumvents the curse of dimensionality: an appreciable property for certain high-dimensional applications. It is often convenient to model such applications as optimization over the set of tensors with fixed (and low) TT rank. That set is a smooth manifold. Exploiting this fact, others have shown that Riemannian optimization techniques can perform particularly well on tasks such as tensor completion and special large-scale linear systems from PDEs. So far, however, these optimization techniques have been limited to first-order methods, likely because of the technical hurdles in deriving exact expressions for the Riemannian Hessian. In this paper, we derive a formula and efficient algorithm to compute the Riemannian Hessian on this manifold. This allows us to implement second-order optimization algorithms (namely, the Riemannian trust-region method) and to analyze the conditioning of optimization problems over the fixed TT rank manifold. In settings of interest, we show improved optimization performance on tensor completion compared to first-order methods and alternating least squares (ALS). Our work could have applications in training of neural networks with tensor layers. Our code is freely available.
\end{abstract}


\section{Introduction}

Tensors of order $d$ are multi-dimensional arrays with some size $n_1 \times \cdots \times n_d$.
They occur in numerous applications, sometimes as the unknown in an optimization problem.
Aside from training neural networks, examples include predicting gene expression (e.g. \citep{iwata2019genomics}) and solving differential equations (e.g. \citep{sergey2019ode}).

In the same way that optimizing over large matrices (tensors of order two) may be challenging, so optimizing over large tensors requires care.
When optimizing over matrices, it is often the case that one can meaningfully restrict attention to matrices of a given low rank $r$.
This may be either because the solution of the problem genuinely is a matrix of rank $r$, or because it can be well approximated by one.
When it comes to tensors, there exist several notions of low rank, with their pros and cons.
We focus on the notion of rank associated with the \emph{tensor train} (TT) format (also known as \emph{matrix product state} (MPS) in the physics community).

The \emph{TT-rank} of a tensor $X$ of order $d$ is a tuple of integers: $\rank_{\mathrm{TT}}\mathrm{(}X\mathrm{)} = \mathbf{r} = \left(r_1, \ldots, r_{d-1}\right)$.
For $d = 2$, it reduces to the usual notion of matrix rank.
We consider optimization problems over the set
\begin{align*}
\Mr & = \{ X \in \reals^{n_1 \times \cdots \times n_d} : \rank_{\mathrm{TT}}\mathrm{(}X\mathrm{)} = \mathbf{r} \}.
\end{align*}
As reviewed below, this set can be endowed with the structure of a Riemannian submanifold of the Euclidean space $\reals^{n_1 \times \cdots \times n_d}$.
This makes it possible to use general techniques from Riemannian optimization~\citep{AMS08} to minimize functions on $\Mr$.

Other authors have exploited the Riemannian structure of $\Mr$ to design first-order optimization algorithms such as gradient descent and certain quasi-Newton schemes~\citep{steinlechner2016riemannian,uschmajew2020geometrylowrank}.
However, no second-order optimization algorithms on $\Mr$ have been implemented yet. Having access to the Riemannian Hessian, we can expect to see superlinear local convergence of algorithms such as the Riemannian trust-region method~\citep{genrtr}. This is manifest in our numerical experiments.

Motivated by the role of tensors in modern machine learning applications (including as a means to encode weights in layers of neural networks) and by the recently revived interest in second-order methods for machine learning tasks,\footnote{See for example the program of the NeurIPS 2019 workshop ``Beyond first order methods in machine learning systems,'' \url{https://sites.google.com/site/optneurips19/}.} in this paper we derive the geometric and numerical tools necessary to implement second-order optimization algorithms on $\Mr$. In particular, we implement the tools necessary to use Riemannian trust-region methods~\citep{genrtr} in the Manopt toolbox~\citep{manopt}. The main ingredient is an efficient procedure to evaluate Riemannian Hessians.

To illustrate the benefits of second-order optimization algorithms on $\Mr$, we experiment with low-rank tensor completion, analogous to the well-known low-rank matrix completion problem.

An alternative encoding for tensors is the \emph{Tucker format}, with its associated notion of Tucker or multilinear rank.
The set of tensors with fixed Tucker rank is also a manifold.
Second-order methods for optimization over that manifold are developed in~\citep{heidel2018riemannian}.
For high-order tensors, the strength of the TT format is that the dimension of $\Mr$ grows linearly in $d$, whereas the dimension of the manifold of fixed Tucker rank grows exponentially in $d$ (the base of this exponential growth is the rank, not the size of the tensors).
The canonical polyadic (CP) format also escapes the curse of dimensionality, but the set of tensors with fixed CP rank is difficult to handle for optimization~\citep[\S9.3]{uschmajew2020geometrylowrank}.

\section{Tensor train format}




Suppose we can factor a tensor $X$ of order $d$ and size $n_1 \times \cdots \times n_d$ into the following form:
\begin{equation}\label{TT_definition}
    X(i_1, \ldots, i_d) = U_1(i_1)U_2(i_2) \cdots U_d(i_d)
\end{equation}
where each $U_\mu(i_\mu) \in \mathbb{R}^{r_{\mu - 1} \times r_\mu}$ is a small matrix for some integers $r_0, \ldots, r_d$. Since $X(i_1, \ldots, i_d)$ is a scalar, we necessarily have $r_0 = r_d = 1$: we collect the remaining sizes in the tuple $\textbf{r} = (r_1, \ldots, r_{d-1})$. For each $\mu \in [d] := \{1, \ldots, d\}$, we stack all of the matrices $U_\mu(1), \ldots, U_\mu(n_\mu)$ together to form a third-order tensor $U_\mu \in \mathbb{R}^{r_{\mu - 1} \times n_\mu \times r_\mu}$. These third-order tensors are called the \textit{cores}, and the set of cores $\{U_1, \ldots, U_d\}$ form a \textit{tensor train decomposition} of $X$. The size $\mathbf{r}$ of the decomposition is $(r_1, \ldots, r_{d-1})$, a vector closely related to the \textit{TT-rank} defined below. A tensor decomposed into this format is a \textit{TT-tensor} \citep{oseledets2011tensortrain}. For example, a TT-tensor of order $d$ and size $n \times \cdots \times n$ with ``constant'' $\mathbf{r} = (r, \ldots, r)$ is fully specified by $O(r^2 n d)$ real numbers (as opposed to $n^d$ in full generality): the linear scaling in $d$ (as opposed to exponential) is how this format escapes the curse of dimensionality \citep{oseledets2009curseofdim}. Of course, the decomposition is not unique.


We define the TT-rank of a tensor $Z \in \mathbb{R}^{n_1 \times \cdots \times n_d}$ by the following:
\begin{align*}
    \rank_\mathrm{TT}(Z) &:= \left(\rank(Z^{<1>}), \ldots, \rank(Z^{<d-1>})\right),
\end{align*}
where each $Z^{<k>}$ is a so-called \emph{flattening} of the tensor into a matrix (see Def. \ref{flattening_def} in Appendix \ref{prelim_defs}).

The following theorem provides support for the latter definition. In essence, it states that $Z$ admits a TT-decomposition with size $\mathbf{r} = \rank_{\mathrm{TT}}(Z)$, but not less. See \cite[Thm. 9.2]{uschmajew2020geometrylowrank} for a proof.
\begin{theorem}\label{minimal_decomp}
Let $Z \in \mathbb{R}^{n_1 \times \cdots \times n_d}$, $Z \ne 0$. Denote $\tilde{r_k} := \rank(Z^{<k>})$. For any tensor train decomposition $U_1, \ldots, U_d$ of size $\mathbf{r} = (r_1, \ldots, r_{d-1})$, it necessarily holds that $r_k \ge \tilde{r_k}$ for all $k \in [d-1]$, and it is furthermore possible to obtain a decomposition such that equality holds.
\end{theorem}
Based on the latter statement, we say a TT decomposition of a tensor is \emph{minimal} if the sizes of the cores match the TT-rank of the tensor.

\section{Smooth manifold structure}
We give a concise overview of the geometry of $\Mr$ restricted to properties useful for optimization.
For book-length introductions to the topic of Riemannian optimization, we direct the reader to~\citep{AMS08,boumal2020intromanifolds}.
For a full treatment of the geometry of fixed TT-rank tensors specifically, we recommend~\citep{kressner2013tensors,steinlechner2016riemannian,uschmajew2020geometrylowrank}.

The set of tensors of size $n_1 \times \cdots \times n_d$ and fixed TT-rank $\mathbf{r} = (r_1, \ldots, r_{d-1})$,
\begin{align*}
	\Mr & = \{ X \in \mathbb{R}^{n_1 \times \cdots \times n_d} : \rank_{\mathrm{TT}}(X) = \mathbf{r} \},
\end{align*}
is a subset of $\mathbb{R}^{n_1 \times \cdots \times n_d}$.
In the same way that the set of matrices of size $m \times n$ and rank $r$ is smoothly embedded in $\mathbb{R}^{m \times n}$, $\Mr$ is a smooth embedded submanifold of $\mathbb{R}^{n_1 \times \cdots \times n_d}$ of dimension 
\begin{align}
   \dim \Mr = \sum_{i=1}^d r_{i-1}n_i r_i - \sum_{i=1}^{d-1} r_i^2.
\end{align}
This means that around each point $X \in \Mr$ we can define a linearization of $\Mr$ called the \emph{tangent space} at $X$.
This is a linear subspace $\T_X\Mr$ of $\mathbb{R}^{n_1 \times \cdots \times n_d}$ which consists in all the ``vectors'' (in fact, tensors) of the form $c'(0)$ where $c(t)$ is a smooth curve in $\mathbb{R}^{n_1 \times \cdots \times n_d}$ which lies entirely on $\Mr$ and passes through $X$ so that $c(0) = X$. Explicitly, given a minimal left-orthogonal decomposition $U_1, \ldots, U_d$ of $X$ (left-orthogonal is defined in Appendix \ref{orthogonal_decomp}), we can parametrize a tangent vector $V$ at $X$ by tensors $\delta V_1, \ldots, \delta V_d$ of the same shape as $U_1, \ldots, U_d$  such that
\begin{align}\label{variation_init_def}
    V(i_1, \ldots, i_d) &= \sum_{k=1}^d U_1(i_1) \cdots U_{k-1}(i_{k-1}) \delta V_k(i_k) U_{k+1}(i_{k+1}) \cdots U_d(i_d)
\end{align}
and $(\delta V_i^L)^\top U_i^L = 0$ for $i \in [d-1]$, where $U_i^L := (U_i)^{<2>}$ and $\delta V_i^L := (\delta V_i)^{<2>}$  (see Def. \ref{LR-flattening} of Appendix \ref{prelim_defs}). The space $\reals^{n_1 \times \cdots \times n_d}$ has the usual Euclidean inner product
\begin{align*}
	\inner{V}{W} & = \sum_{i_1, \ldots, i_d} V(i_1, \ldots, i_d) W(i_1, \ldots, i_d).
\end{align*}
We equip each tangent space with the same inner product simply by restricting the domain.
This turns $\Mr$ into a Riemannian manifold; specifically: a Riemannian submanifold of $\mathbb{R}^{n_1 \times \cdots \times n_d}$.

A function $f \colon \Mr \to \mathbb{R}$ is smooth if and only if it is the restriction of a smooth function $\bar f$ defined on a neighborhood of $\Mr$ in $\mathbb{R}^{n_1 \times \cdots \times n_d}$. The Riemannian structure affords us a notion of gradient and Hessian for $f$, of central importance for optimization.
Specifically, the \emph{Riemannian gradient} of $f$ at $X \in \Mr$ is the (unique) tangent vector $\grad f(X) \in \T_X\Mr$ such that
\begin{align*}
	\forall V \in \T_X \Mr, && \inner{\grad f(X)}{V} & = \lim_{t \to 0} \frac{\bar f(X + tV) - \bar f(X)}{t}.
\end{align*}
It can be shown that this does not depend on the choice of smooth extension $\bar f$.
If $\calP_X$ denotes the orthogonal projector from $\mathbb{R}^{n_1 \times \cdots \times n_d}$ to $\T_X \Mr$, it is easy to verify that
\begin{align*}
	\grad f(X) & = \calP_X\!\left( \partial \bar f(X) \right),
\end{align*}
where $\partial \bar f(X)$ is the (classical) gradient of $\bar f$ at $X$.
Provided $\partial \bar f(X)$ is sufficiently structured, this can be computed efficiently.

A \emph{retraction} is a smooth map on the tangent bundle which provides maps $\R_X \colon \T_X\Mr \to \Mr$ such that $c(t) = \R_X(tV)$ is a smooth curve on $\Mr$ satisfying $c(0) = X$ and $c'(0) = V$.
For example, a computationally favorable choice for $\Mr$ is the \textit{TT-SVD}~\citep{oseledets2011tensortrain},\cite[\S 9.3.4]{uschmajew2020geometrylowrank}.


Combined, the tools described here are sufficient to develop first-order optimization methods on $\Mr$, including Riemannian gradient descent and even some quasi-Newton methods \citep{steinlechner2016riemannian}. However, to implement true-to-form second-order optimization methods, we also need access to the Riemannian Hessian: this is our main object of study. But first, we need a second look at tangent vectors.

\paragraph{An alternative parametrization of $\T_X \Mr$.}
Orthogonal projections and inner products of tangent vectors are computed frequently in optimization algorithms, so it is key to have a parametrization of $\T_X\Mr$ that yields efficient computation of both. The following parametrization of $\T_X\Mr$ was first proposed in \citep{khoromskij2012ttmanifold}, and further elaborated on in \citep{steinlechner2016riemannian}. Given a tangent vector represented by $\delta V_1, \ldots, \delta V_d$ (as in eq. \eqref{variation_init_def}), we can generate another representation $\delta \widetilde{V}_1, \ldots, \delta \widetilde{V}_d$  such that $(\delta \widetilde{V}_i^L)^\top U_i^L = 0$ for $i \in [d-1]$ and such that the tangent vector $V$ is given by:
\begin{align}\label{param_alt}
	V(i_1, \ldots, i_d) &= \sum_{k=1}^d U_1(i_1) \cdots U_{k-1}(i_{k-1}) \delta \widetilde{V}_k(i_k) \widetilde{U}_{k+1}(i_{k+1}) \cdots \widetilde{U}_d(i_d)
\end{align}
where $\{\widetilde{U}_k \}$ are the right-orthogonalized cores from $\{U_k \}$ (see Appendix \ref{orthogonal_decomp}). Conversely, we can also recover $\delta V_1, \ldots, \delta V_d$ from $\delta \widetilde{V}_1, \ldots, \delta \widetilde{V}_d$. The inner product of two tangent vectors $V$, $W$ with parametrizations $\delta \widetilde{V}_1, \ldots, \delta \widetilde{V}_d$, and $\delta \widetilde{W}_1, \ldots, \delta \widetilde{W}_d$ admits a convenient expression:
\begin{align*}
    \inner{V}{W} &= \sum_{i=1}^d \inner{\delta \widetilde{V}_i}{\delta \widetilde{W}_i}.
\end{align*}
This is computable in $O(dnr^2)$ flops, where $r := \max_\mu r_\mu$ and $n := \max_\mu n_\mu$. Importantly for our purpose, the tangent space can be decomposed into $d$ orthogonal subspaces \citep{steinlechner2016riemannian}, so we can decompose the orthogonal projector $\P_X$ into $d$ orthogonal components:
\begin{align}\label{proj_split}
    \calP_X = \P_X^1 + \cdots + \P_X^d
\end{align}
where $\P_X^1, \ldots, \P_X^d$ are the projectors to the orthogonal subspaces and are given by the following:
\begin{align}\label{proj_explicit}
    (\P_X^k(Z))^{<k>} = (I_{n_k} \otimes X_{\le k-1})\left( (I_{n_kr_{k-1}} - U_k^L(U_k^L)^\top)(I_{n_k} \otimes X_{\le k-1}^\top)Z^{<k>}\Xt_{\ge k+1}\right)\Xt_{\ge k+1}^\top
\end{align}
for $k \in [d-1]$, and
\begin{align}\label{proj_d}
    (\P_X^d(Z))^{<d>} = (I_{n_d} \otimes X_{\le d-1})(I_{n_d} \otimes X_{\le d-1}^\top) Z^{<d>},
\end{align}
where $X_{\le k-1}$, $\Xt_{\ge k+1}$ are so-called \textit{interface matrices} from decompositions $\{U_k\}$ and $\{\widetilde{U}_k\}$ respectively (see Def. \ref{flattening_def} in Appendix \ref{prelim_defs}). We use parametrization $\{\delta V_k\}$ to represent tangent vectors in differentials, and $\{\delta \widetilde{V}_k\}$ for the resulting tangent vector after orthogonal projection to the tangent space. We prove interchangeability between these parametrizations in Lemma \ref{interchange_tangent_param} of Appendix \ref{alternative}.


\section{Riemannian Hessian}
The Riemannian Hessian of $f$ at $X$---a symmetric operator to and from $\T_X\Mr$---admits an explicit expression in terms of the Euclidean derivatives of $\bar f$ at $X$.
It is shown in~\citep{absil2013extrinsic} for general Riemannian submanifolds that, for all $V \in \T_X\Mr$,
\begin{align}\label{riemannian_hessian}
	\Hess f(X)[V] & = \P_X \partial^2 \bar{f}(X)[V] + \P_X (\D_V \P_X) \partial \bar{f}(X).
\end{align}
A few comments are in order. For the first term, $\partial^2 \bar{f}(X)[V]$ is the Euclidean Hessian of $\bar f$ at $X$ along $V$, the result of which is then projected to $\T_X\Mr$ through $\P_X$.
The second term is a ``correction term'' in the sense that it modifies the (projected) Euclidean Hessian to capture the Riemannian geometry of $\Mr$.
The notation $\D_V\P_X$ denotes the differential of the map $X \mapsto \P_X$ at $X$ along the direction $V$, so that
\begin{align*}
	(\D_V \P_X) Z =\D_V(\calP_X Z) = \lim_{t \to 0} \frac{\P_{c(t)}(Z) - \P_{c(0)}(Z)}{t},
\end{align*}
where $c(t)$ is any smooth curve on $\Mr$ such that $c(0) = X$ and $c'(0) = V$.
In words: it is the derivative of the orthogonal projector to $\T_X\Mr$ as we perturb $X$ along the tangent direction $V$.
As such, $\D_V\P_X$ is itself a linear operator from $\mathbb{R}^{n_1 \times \cdots \times n_d}$ to $\mathbb{R}^{n_1 \times \cdots \times n_d}$.
As shown in~\citep{absil2013extrinsic}, the correction term depends only on the normal component of $\partial \bar{f}(X)$.
Moreover, the operation which maps a tangent vector $V$ and a normal vector $N$ to the tangent vector $\P_X (\D_V\P_X) N$ is the \emph{Weingarten map}: a standard object in geometry.

Note that splitting $\P_X = \sum_{i=1}^d \P_X^i$ as in eq. \eqref{proj_split} allows us to rewrite the correction term as:
\begin{align}\label{eq_split_correction}
    \P_X (\D_V \P_X) (Z) & = \sum_{k=1}^d \P_X^k (\D_V \P_X^k) (Z) + \sum_{i = 1}^d \sum_{j = 1, j \neq i}^d \P_X^i (\D_V \P_X^j) (Z).
\end{align}
The double sum would seem to take too many flops to compute. However, we show in this paper an expression for these ``cross-terms'' $\P_X^i (\D_V \P_X^j) (Z), i \ne j$ in a way that yields a computation of the whole double sum in virtually no extra flops after computing $\sum_{k=1}^d \P_X^k (\D_V \P_X^k) (Z)$, which we do efficiently.

We now present the main contribution of this paper: simplified formulas for the correction term that yield computationally efficient algorithms for the Hessian. Proofs for these formulas can be found in Appendix \ref{proofs_correction}, and time complexity analyses can be found in Appendix \ref{proofs_complexity}. See Appendix \ref{orthogonal_decomp} for definitions of the small invertible matrices $R_k \in \reals^{r_k \times r_k}$ and matrices $\Xt_{\ge k+1} \in \reals^{n_{k+1}\cdots n_d \times r_k}$. We introduce the matrices $V_{< k}$ and $V_{> k}$ as the \textit{variational interface matrices} of tangent vector $V$, defined in Def. \ref{variational_interface} of Appendix \ref{prelim_defs}.

\begin{theorem}\label{TT_diagonal_terms}
    The terms in the first sum of eq. \eqref{eq_split_correction} can be computed as follows. For $k < d$, we have:
    \begin{align*}
        (\P_X^k \De_V \P_X^k Z)^{<k>} &= (I \otimes X_{< k}) \left(((I - U_k^L(U_k^L)^\top)(I \otimes V_{< k}^\top) )Z^{<k>} \Xt_{\ge k+1} \right) \Xt_{\ge k+1}^\top \\
        &- (I \otimes X_{< k}) \left(\delta V_k^L (U_k^L)^\top (I \otimes X_{< k}))Z^{<k>} \Xt_{\ge k+1} \right) \Xt_{\ge k+1}^\top \\
        + (I \otimes X_{< k})( (&I - U_k^L(U_k^L)^\top)(I_{n_k} \otimes X_{\le k-1}^\top)Z^{<k>}(I - \Xt_{\ge k+1} \Xt_{\ge k+1}^\top) V_{\ge k+1} R_k^{-1})\Xt_{\ge k+1}^\top
    \end{align*}
    while for $k=d$ we have
    \begin{equation*}
        (\P_X^d \De_V \P_X^d Z)^{<d>} = (I \otimes X_{\le d-1})(I \otimes V_{< d}^\top) Z^{<d>}.
    \end{equation*}
    For typically structured $Z$ (e.g., full, sparse, low TT-rank), these formulas yield algorithms for computing the first sum in a number of arithmetic operations similar to that required for the computation of $\calP_X(Z)$.
\end{theorem}

For the cross-terms (non-diagonal terms), we show the identity $\P_X^i \D_V \P_X^j Z = -(\De_V \P_X^i) \P_X^j Z$. This allows us to express the cross-terms in not only a simplified way, but into an expression that allows us to re-use computations already made for the diagonal terms:
\begin{theorem}\label{TT_cross_terms}
    The terms in the double sum of eq. \eqref{eq_split_correction} can be computed as follows.
    \begin{align*}
        j > i, i < d : (\P_X^i \De_V \P_X^j Z)^{<i>} &= (I \otimes X_{< i}) \left(\delta V_i^L (Y_{\ge i+1}^j)^\top \Xt_{> i} \right) \Xt_{> i}^\top, \\
        j < i, i < d : (\P_X^i \De_V \P_X^j Z)^{<i>} &= -(I \otimes X_{< i})\left( (I - U_k^L(U_k^L)^\top)(I \otimes V_{< i}^\top)Y^j_{\le i} \right) \Xt_{\ge i+1}^\top, \\
        j < i, i = d : (\P_X^d \De_V \P_X^j Z)^{<d>} &= -(I \otimes X_{< d})(I \otimes V_{< d}^\top)Y^j_{\le d},
    \end{align*}
    where $Y^j$ is the TT-tensor given by $\P_X^j Z$. Given the set $\{Y^1, \ldots, Y^n\}$, which are computed as a by-product from efficient algorithms for the first sum, the cross-terms $\sum_{i = 1}^d \sum_{j = 1, j \neq i}^d \P_X^i (\D_V \P_X^j) (Z)$ are computable in $O(d^2nr^3)$ arithmetic operations.
\end{theorem}

\section{Numerical analysis on tensor completion}

Using our analytical expression of the Riemannian Hessian, we develop a Riemannian Trust Regions (RTR) method for solving optimization problems over $\Mr$ \citep{genrtr, manopt}. To assess the performance of this method, we compare RTR with Alternating Least Squares (ALS) and a conjugate gradient method on tensor completion (RTTC) \cite{Steinlechner2016RiemannianOF}, both of which were coded by Steinlechner et al.\ We also compare to RTR when we use a finite-difference approximation of the Riemannian Hessian: we denote the resulting algorithm FD-TR.

In our experiments, all tensors are of size $(4, 4, \ldots, 4)$ and some order $d$, specified at each experiment. For each experiment, we report the convergence of each algorithm in terms of cost (training cost), test cost from an independent set of samples, and gradient norm for algorithms where this applies (RTR, FD-TR, RTTC). A quantity of critical importance is the \textit{oversampling ratio}: $|\Omega| / \text{dim}(\Mr)$, where $|\Omega|$ is the number of observed indices. We also report the \textit{sampling ratio}, $|\Omega| / 4^d$.

Graphs of the results and further details of the experiments can be found in Appendix \ref{experiments}. In summary, we find that RTR is slower on versions of tensor completion with better conditioned Hessians, but outperforms other algorithms in more challenging instances of the problem where the target point Hessian has worse conditioning (something we can assess using our Hessian formulas).

\subsection*{Funding}
This work was supported by the National Science Foundation through award DMS-1719558.

\bibliography{boumal}


\pagebreak
\appendix
\section{Experiments on tensor completion}\label{experiments}

We consider target points randomly generated on the manifold by constructing normally distributed TT-cores. Observed entries are chosen according to some distribution $\textbf{p} = (p_1, p_2, p_3, p_4)$ such that, for each sample index $(i_1, \ldots, i_d$), each $i_k$ is chosen at random from $\{1, 2, 3, 4\}$ according to the distribution $\textbf{p}$ (not necessarily uniform). These experiments illustrate the observation that second-order methods perform better on ``harder'' versions of tensor completion.

For each figure, we plot the (training loss/test loss/gradient norm) over 10 trials, each trial with a different random initialization and target tensor. The three algorithms are compared on each trial, and the 10 trials are then plotted over each other in a single chart. The differences in the problem setting for each figure are described in the figure details.

For Trust Regions, we used a starting radius of $100$ and a maximum radius of $100 \cdot 2^{11}$.
\vspace{4mm}
\begin{figure}[H]
    \centering
    \includegraphics[width=\columnwidth]{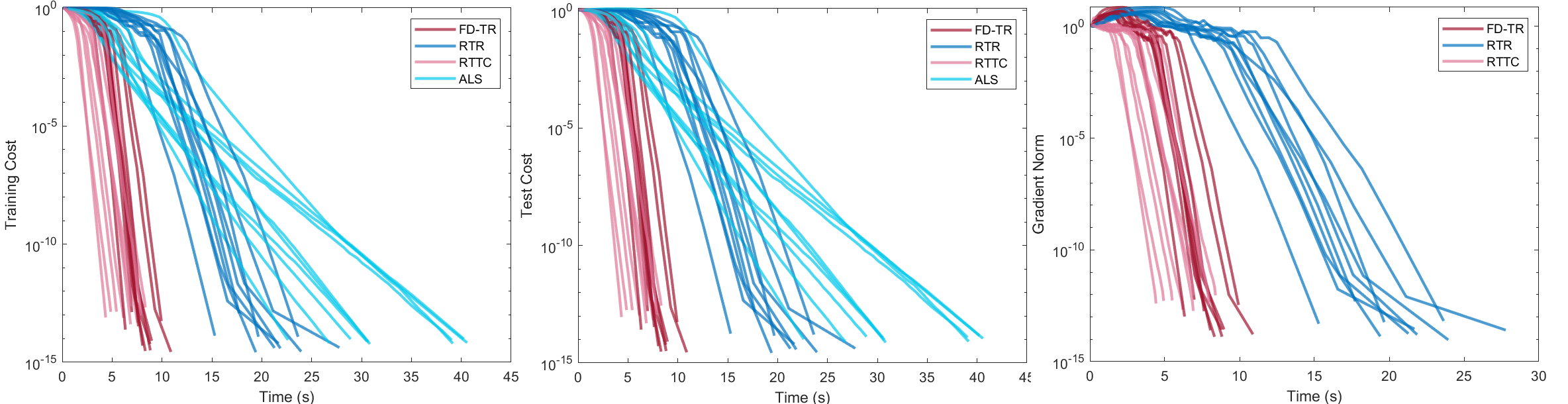}
    \caption{$d = 9$, $\mathbf{r} = (3, 5, 10, 10, 10, 10, 5, 3)$, oversampling ratio: 20.5, sampling ratio: 0.1, $\mathbf{p} = (\frac{1}{4}, \frac{1}{4}, \frac{1}{4}, \frac{1}{4})$, condition numbers of the Hessian at each target point: around 10. This could be considered as an ``easier'' tensor completion problem, as both the oversampling and sampling ratios are relatively high, and entries are sampled uniformly at random.}
\end{figure}
\begin{figure}[H]
    \centering
    \includegraphics[width=\columnwidth]{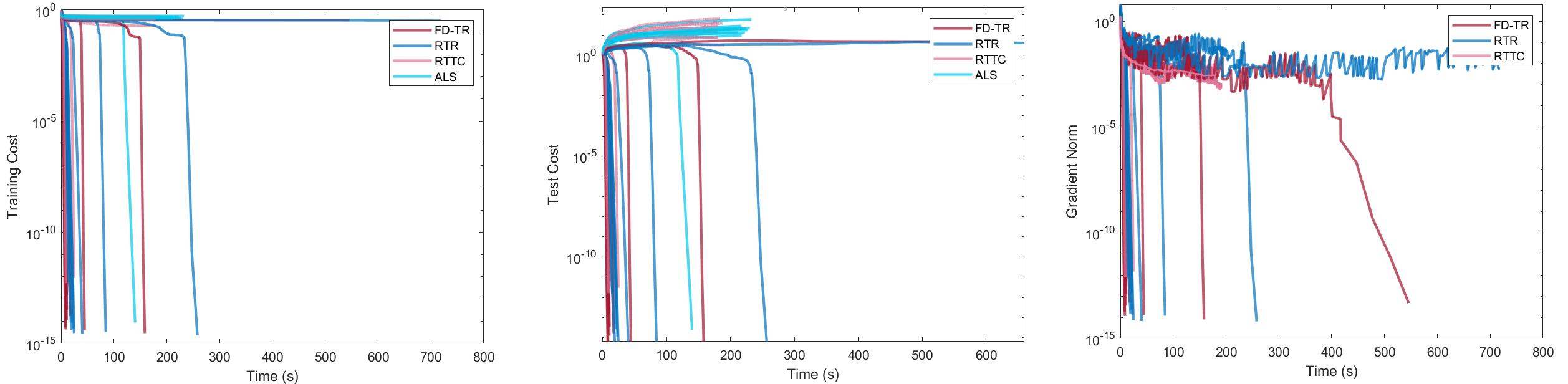}
    \caption{$d = 9$, $\mathbf{r} = (3, 4, 8, 12, 12, 8, 4, 3)$, Oversampling Ratio: 5.2, Sampling ratio: 0.025, $\mathbf{p} = (\frac{2}{5}, \frac{1}{5}, \frac{1}{5}, \frac{1}{5})$, Hessian condition numbers: $10^2 - 10^3$. As we decrease the oversampling ratio, the benefits of second-order methods start to show. In the alloted time, out of 10 trials, ALS converged once, RTTC converged twice, and both RTR and FD-TR converged 9 times. Note though that FD-TR converges faster than RTR; FD-TR performs similarly to RTR iteration-wise, but the analytic Hessian takes longer to compute than the finite difference approximation.}
\end{figure}
\begin{figure}[H]
    \centering
    \includegraphics[width=\columnwidth]{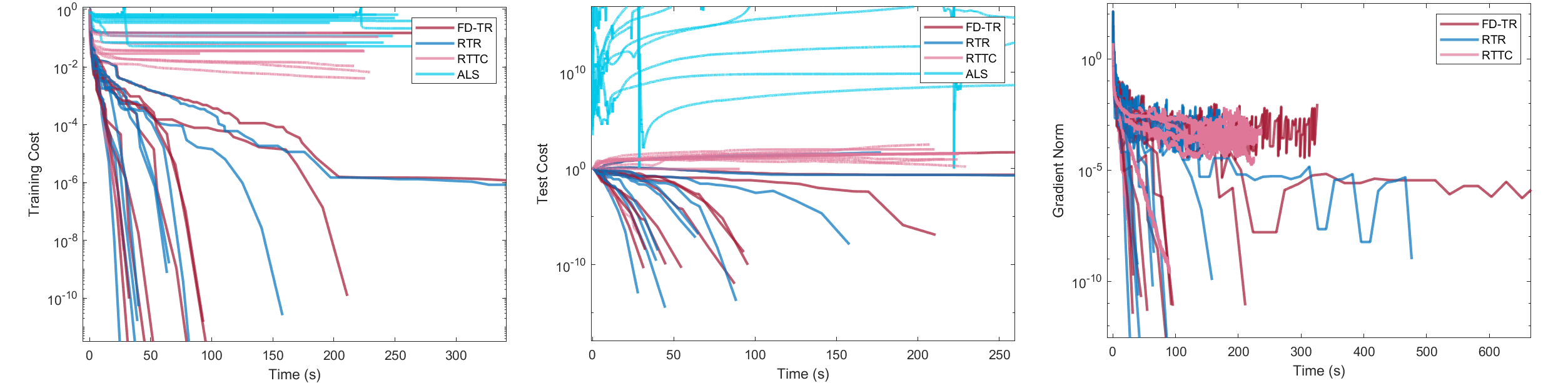}
    \caption{$d = 9$, $\mathbf{r} = (2, 2, 3, 3, 3, 3, 2, 2)$, Oversampling Ratio: 5.1, Sampling ratio: 0.003, \textbf{p} = $(\frac{50}{65}, \frac{12}{65}, \frac{2}{65}, \frac{1}{65})$, Hessian condition numbers: $10^5 - 10^7$. This example illustrates that having access to the true Hessian can yield an advantage over using a finite difference approximation. RTR outperformed all other algorithms in each instance separately.}
\end{figure}

\section{Proofs for the tensor train format}

This section contains proofs for properties of tensors in the tensor train format. We use these initial results to work out the main result of the paper: the Riemannian Hessian for $\Mr$.
\subsection{Preliminary definitions}\label{prelim_defs}
In this subsection, we establish the various notation used in subsequent proofs. For any tensor $X \in \reals^{n_1 \times \cdots \times n_d}$, it holds that $X \in \Mr$ if and only if $(\rank(X^{<1>}), \ldots, \rank(X^{<d-1>})) = \mathbf{r}$. For $\Mr$ to be non-empty, it is necessary and sufficient for $r_{k-1} \le n_k r_k$ and $r_k \le n_k r_{k-1}$ for all $k \in [d]$, so for the remainder of this section, we assume these conditions hold. This statement can be found in \citep[eq. ~(9.32)]{uschmajew2020geometrylowrank}.

For any $X \in \Mr$, unless otherwise stated, denote $U_1, \ldots, U_d$ to be a minimal, left-orthogonal TT-decomposition (note this decomposition is not unique), and let $\widetilde{U}_1, \ldots, \widetilde{U}_d$ be the resulting TT-decomposition from right-orthogonalization of $U_1, \ldots, U_d$ using \cite[Alg.~4.1]{steinlechner2016riemannian}. Let $X_{\le k}, X_{\ge k}$ denote the interface matrices for $U_1, \ldots, U_d$, and let $\Xt_{\le k}, \Xt_{\ge k}$ denote the interface matrices for $\widetilde{U}_1, \ldots, \widetilde{U}_d$.

\begin{definition}\label{flattening_def}
Let $Z \in \mathbb{R}^{n_1 \times \cdots \times n_d}$ be a tensor of order $d$. Then the ``$\mu$th flattening'', written as $Z^{<\mu>} \in \mathbb{R}^{n_1\cdots n_\mu \times n_{\mu + 1}\cdots n_d}$ flattens $Z$ to a matrix of size $(i_1 \cdots i_\mu \times i_{\mu + 1} \cdots i_d)$: within each dimension of the matrix, the indices are ordered colexicographically. (This is the same as calling Matlab's ``reshape'' method on a multidimensional array with the specified target dimension.)
\end{definition}
    
\begin{definition}\label{LR-flattening}
Let $U_\mu$ be a core of some TT-decomposition of $X$. Since cores are third-order tensors, there are only two non-vector $\mu$th flattenings: $U_\mu ^{<1>} \in \mathbb{R}^{r_{\mu - 1} \times n_\mu r_\mu}$ and $U_\mu ^{<2>} \in \mathbb{R}^{r_{\mu - 1}n_\mu \times r_\mu}$, which we will denote $U_\mu^R$ and $U_\mu^L$ respectively. These are called the ``right flattening'' and ``left flattening''. Without loss of generality, we assume that $(U_\mu^L)^\top U_\mu^L = I_{r_\mu}$ for all $\mu \in \{1, \ldots, d-1\}$ (see \cite[\S 4.2.1]{steinlechner2016riemannian}).
\end{definition}
\begin{definition}\label{def:interface}
    Let $X$ be a TT-tensor with cores $U_1, \ldots, U_d$. For each $k \in \{0, \ldots, d\}$, define $X_{\le k} \in \reals^{n_1 \cdots n_k \times r_k}$ by the recursive formula $X_{\le k} = (I_{n_k} \otimes X_{\le k-1})U_k^L$ and base case $X_{\le 0} = 1 \in \reals$. Similary, for each $k \in \{1, \ldots, d+1\}$, define $X_{\ge k} \in \reals^{r_{k-1} \times n_k \cdots n_d }$ by the recursive formula $X_{\ge k}^\top = U_k^R (X_{\ge k+1}^\top \otimes I_{n_k})$ and the base case $X_{\ge d+1} = 1 \in \reals$. We call these matrices the \textit{interface matrices} of $X$. While phrased differently, this definition is equivalent to the one given in \citep[\S 4.1]{steinlechner2016riemannian}.
\end{definition}
\begin{definition}\label{variational_interface}
    Let $V \in \T_X \Mr$ be represented as in eq. \eqref{variation_init_def}. For each $k \in \{0, \ldots, d\}$, define $V_{\le k} \in \reals^{n_1 \cdots n_k \times r_k}$ by the recursive formula $V_{\le k} = (I_{n_k} \otimes V_{\le k-1})U_k^L + (I_{n_k} \otimes X_{\le k-1}){\delta V}_k^L$ and base case $V_{\le 0} = 0 \in \reals$. Similary, for each $k \in \{1, \ldots, d+1\}$, define $V_{\ge k} \in \reals^{r_{k-1} \times n_k \cdots n_d }$ by the recursive formula $V_{\ge k}^\top = U_k^R (V_{\ge k+1}^\top \otimes I_{n_k}) + \delta V_k^R (X_{\ge k+1}^\top \otimes I_{n_k})$ and the base case $V_{\ge d+1} = 0 \in \reals$. We call these matrices the \textit{variational interface matrices} of tangent vector $V$.
\end{definition}

\subsection{$\mu$-orthogonal decompositions}\label{orthogonal_decomp}

Let $X \in \mathbb{R}^{n_1 \times \cdots \times n_d}$ be a TT-tensor with a minimal decomposition $U_\mu$. We call a decomposition \textit{$\mu$-orthogonal} if $(U_k^L)^\top U_k^L = I_{r_k}$ for all $k \in [\mu - 1]$ and $U_k^R (U_k^R)^\top = I_{r_{k-1}}$ for all $k \in \{\mu + 1, \ldots , d\}$. We also call $d$-orthogonal decompositions \textit{left orthogonal} and $1$-orthogonal decompositions \textit{right orthogonal}. Transforming a decomposition to a $\mu$-orthogonal one without changing the underlying tensor is called \textit{$\mu$-orthogonalization}, and any minimal TT-decomposition can be $\mu$-orthogonalized for any $\mu \in [d]$ in $O(dnr^3)$ flops using the algorithm presented in \citep[Alg.~4.1]{steinlechner2016riemannian}.

We will frequently use the right-orthogonalization of a minimal left-orthogonal decomposition $U_1, \ldots, U_d$; denote these resulting right-orthogonal cores $\widetilde{U}_1, \ldots, \widetilde{U}_d$ and their interface matrices $\widetilde{X}_{\le k}$ and $\widetilde{X}_{\ge k}$. Note this is a different definition for $\widetilde{X}_{\ge k}$ from what is given for Theorem \ref{TT_diagonal_terms}; these two definitions for $\widetilde{X}_{\ge k}$ indeed define the same matrix \citep[\S 4.2.1]{steinlechner2016riemannian}. We also define $\{R_2, \ldots, R_d\}$ by the relation $X_{\ge k+1} = \Xt_{\ge k+t} R_k$, which are generated as a by-product of \citep[Alg.~4.1]{steinlechner2016riemannian}.

\subsection{Proofs regarding the Tensor Train format}
The following two lemmas can be found in \citep[\S 4.1-2]{steinlechner2016riemannian} and are fundamental identities for later proofs.

\begin{lemma}\label{interface:rows}
    Let $X \in \Mr$, and recall Definition \ref{def:interface}. Let $\mathbf{i} = (i_1, \ldots, i_d)$ be a multi-index for $X$. The following two identities hold:
    \begin{enumerate}
    \item $\emph{row}_{i_1 \cdots i_k}\left(X_{\le k}\right) = U_1(i_1) \cdots U_k(i_k)$\label{TT_parttwo}
    \item $\emph{row}_{i_1 \cdots i_k}\left(X_{\ge k}\right) = \left(U_k(i_k) \cdots U_d(i_d)\right)^\top$ \label{TT_partthree}
    \end{enumerate}
\end{lemma}
where $\text{row}_{\mu}(A) \in \reals^{1 \times n}$ denotes the $\mu$th row of matrix $A \in \reals^{m \times n}$. Lemma \ref{interface:rows} is proven trivially from induction on the inductive definitions for the interface matrices. Using Lemma \ref{interface:rows}, it is straightforward to establish the following equation.
\begin{lemma} \label{TT}
    For any TT-decomposition $U_1, \ldots, U_d$ of $X$ (not necessarily left-orthogonal), the following identity holds:

    $$X^{<k>} = X_{\le k} X_{\ge k+1}^\top$$
\end{lemma}

\begin{lemma} \label{TT-orthogonal}
  If TT-cores $U_1, \ldots, U_d$ are left-orthogonal, the interface matrices $X_{\le 1}, \ldots , X_{\le d-1}$ have orthonormal columns: $X_{\le k}^\top X_{\le k} = I_{r_k}$. Similarly, the interface matrices $\widetilde{X}_{\ge 2}, \ldots , \widetilde{X}_{\ge d}$ have orthonormal columns.
\end{lemma}

\begin{proof}
  We prove by induction, starting with the left-orthogonal case. For the base case, note that $X_{\le 1} = U_1^L$, and we know that $(U_1^L)^\top U_1^L = I_{r_1}$ by left-orthogonality, proving the base case. Now assuming that $X_{\le k-1}^\top X_{\le k-1} = I_{r_{k-1}}$, we expand the matrix $X_{\le k}^\top X_{\le k}$ inductively:
  \begin{align*}
    X_{\le k}^\top X_{\le k} &= (U_k^L)^\top(I_{n_k} \otimes X_{\le k-1}^\top)(I_{n_k} \otimes X_{\le k-1})U_k^L \\
    &= (U_k^L)^\top(I_{n_k} \otimes X_{\le k-1}^\top X_{\le k-1})U_k^L \\
    &= (U_k^L)^\top U_k^L \\
    &= I_{r_k}
  \end{align*}
  This concludes the proof for the left-orthogonal case; the right-orthogonal case can be proven in a similar manner.
\end{proof}

\subsection{Deriving formulas for cores from formulas of flattenings}

Recall from eq. \eqref{eq_split_correction} that we aim to find a formula for $\P_X^i (\D_V \P_X^j) (Z)$, which is a tensor in the $i$th orthogonal component of $\T_X\Mr$. We can then represent $\P_X^i (\D_V \P_X^j) (Z)$ by a variational core $\delta \widetilde{V}_i$, given by eq. \eqref{param_alt}. Using Lemma \ref{TT} and the inductive definition of $X_{\le k}$, we see that eq. \eqref{param_alt} is equivalent to the following:
\begin{align}
    (\P_X^i (\D_V \P_X^j) (Z))^{<i>} = (I_{n_i} \otimes X_{\le i-1}) \delta \widetilde{V}_i^L \Xt_{\ge i+1}^\top.
\end{align}

We would then hope that if we can find an equation for $(\P_X^i (\D_V \P_X^j) (Z))^{<i>}$ in the form $(I_{n_i} \otimes X_{\le i-1}) B \Xt_{\ge i+1}^\top$, where $B \in \reals^{r_{i-1}n_i \times r_i}$, that we have uniqueness: $B = \delta \widetilde{V}_i^L$. This is indeed the case.

\begin{lemma}\label{core_formula}
If a matrix $B \in \reals^{n_ir_{i-1} \times r_i}$ satisfies $(\P_X^i (\D_V \P_X^j) (Z))^{<i>} = (I_{n_i} \otimes X_{\le i-1})B\Xt_{\ge i+1}^\top$, then $B = \delta \widetilde{V}_i^L$.
\end{lemma}
\begin{proof}
    From Lemma \ref{TT}, $(\P_X^i (\D_V \P_X^j) (Z))^{<i>} = (I_{n_i} \otimes X_{\le i-1}) \delta \widetilde{V}_i^L \Xt_{\ge i+1}^\top$. Supposing that a small matrix $B$ satisfies $(\P_X^i(Z))^{<i>} = (I_{n_i} \otimes X_{\le i-1})B\Xt_{\ge i+1}^\top$, we then have:
  \begin{align*}
    (I_{n_i} \otimes X_{\le i-1}) \delta \widetilde{V}_i^L \Xt_{\ge i+1}^\top &= (I_{n_i} \otimes X_{\le i-1})B\Xt_{\ge i+1}^\top \\
    (I_{n_i} \otimes X_{\le i-1}^\top X_{\le i-1}) \delta \widetilde{V}_i^L \Xt_{\ge i+1}^\top\Xt_{\ge i+1} &= (I_{n_i} \otimes X_{\le i-1}^\top X_{\le i-1})B\Xt_{\ge i+1}^\top\Xt_{\ge i+1} \\
    \delta \widetilde{V}_i^L &= B.
  \end{align*}
\end{proof}

Thus, finding a formula for $(\P_X^i (\D_V \P_X^j) (Z))^{<i>}$ in the form $(\P_X^i (\D_V \P_X^j) (Z))^{<i>} = (I_{n_i} \otimes X_{\le i-1})B\Xt_{\ge i+1}^\top$ is sufficient to find a formula for $\delta \widetilde{V}_k$.

\section{Proofs for alternative parametrization of $\T_X\Mr$}
\label{alternative}

\begin{lemma}\label{interchange_tangent_param}
    For a tangent vector $V \in \T_X \Mr$, $\delta \widetilde{V}_k^L = \delta V_k^L R_k^\top$, for all $k \in [d-1]$, where $R_k \in \reals^{r_k \times r_k}$ are defined by $X_{\ge k+1} = R_k \widetilde{X}_{\ge k+1}$, are invertible, and are generated by right-orthogonalization of $U_1, \ldots, U_d$. Lastly, $\delta \widetilde{V}_d^L = \delta V_d^L$.
\end{lemma}

\begin{proof}
    Note that the last statement comes trivially from comparing definitions of the parametrizations. Furthermore, all properties of $R_k$ come from \citep[Algorithm 4.1]{steinlechner2016riemannian}.
    
    Let $V = \P_X(Z)$ for an arbitrary tensor $Z$. Note the following equality between both parametrizations using Lemma \ref{TT}:
    \begin{align*}
    (\P_X^k(Z))^{<k>} = (I_{n_k} \otimes X_{\le k-1}) \delta \widetilde{V}_k^L \tilde{X}_{\ge k+1}^\top = (I_{n_k} \otimes X_{\le k-1}) \delta V_k^L X_{\ge k+1}^\top.
    \end{align*}
    Using the relation $X_{\ge k+1} = \tilde{X}_{\ge k+1} R_{k}$, we then get that:

    \begin{align*}
        (I_{n_k} \otimes X_{\le k-1}) \delta \tilde{V} \tilde{X}_{\ge k+1}^\top &= (I_{n_k} \otimes X_{\le k-1}) \delta V_k^L R_{k+1}^\top \tilde{X}_{\ge k+1}^\top \\
        \delta \tilde{V}_k^L &= \delta V_k^L R_{k+1}^\top
    \end{align*}
    by Lemma \ref{core_formula}.
\end{proof}

Finally, we show a uniqueness result regarding formulas for $\delta \widetilde{V}_k$ from $(\P_X^k(Z))^{<k>}$:

\begin{lemma}\label{core_formula}
    If a small matrix $B \in \reals^{n_ir_{i-1} \times r_i}$ satisfies $(\P_X^k(Z))^{<k>} = (I_{n_i} \otimes X_{\le i-1})B\Xt_{\ge i+1}^\top$, then $B = \delta \widetilde{V}_i^L$.
\end{lemma}
\begin{proof}
    From Lemma \ref{TT}, $(\P_X^k(Z))^{<k>} = (I_{n_i} \otimes X_{\le i-1}) \delta \widetilde{V}_i^L \Xt_{\ge i+1}^\top$. Supposing that a small matrix $B$ satisfies $(\P_X^i(Z))^{<i>} = (I_{n_i} \otimes X_{\le i-1})B\Xt_{\ge i+1}^\top$, we then have:
    \begin{align*}
    (I_{n_i} \otimes X_{\le i-1}) \delta \widetilde{V}_i^L \Xt_{\ge i+1}^\top &= (I_{n_i} \otimes X_{\le i-1})B\Xt_{\ge i+1}^\top \\
    (I_{n_i} \otimes X_{\le i-1}^\top X_{\le i-1}) \delta \widetilde{V}_i^L \Xt_{\ge i+1}^\top\Xt_{\ge i+1} &= (I_{n_i} \otimes X_{\le i-1}^\top X_{\le i-1})B\Xt_{\ge i+1}^\top\Xt_{\ge i+1} \\
    \delta \widetilde{V}_i^L &= B,
    \end{align*}
    which completes the proof.
\end{proof}

\section{Proofs for the correction term}\label{proofs_correction}
In this section, we prove an explicit formula for the Riemannian Hessian of $\Mr$. We take the definition of the Riemannian Hessian to be that given in eq. \eqref{riemannian_hessian}.
\subsection{Extending $(\D_V \P_X^i) (Z)$ to a standard differential} \label{section_diff}

We start from the differential of the full orthogonal projector, $(\D_V \P_X) (Z)$. Recall the definition of $(\D_V \P_X) (Z)$ from the main paper:
\begin{align}\label{differential_original}
	(\D_V \P_X) Z = \lim_{t \to 0} \frac{\P_{c(t)}(Z) - \P_{c(0)}(Z)}{t}
\end{align}
where $c(t)$ is a smooth curve on $\Mr$ such that $c(0) = X$ and $c'(0) = V$. Recall from eq. \eqref{proj_split} that, given a left-orthogonal and minimal decomposition $U_1, \ldots, U_d$ of $X$, we can split $\P_X(Z)$ into the sum $\sum_{i=1}^d\P_X^i(Z)$, where each $\P_X^i(Z)$ is given by eq. \eqref{proj_explicit} and \eqref{proj_d}. Thus, to be able to make this split over the curve $c(t)$, we want to construct $c(t)$ as a ``curve of cores'' in the following way.

Note that each of the cores $U_k$ of a TT-decomposition lives in the linear space $\reals^{r_{k-1} \times n_k \times r_k}$. We construct a smooth map from $(\reals^{r_0 \times n_1 \times r_1}) \times (\reals^{r_1 \times n_2 \times r_2}) \times \cdots \times (\reals^{r_d \times n_d \times r_{d+1}})$ to $\reals^{n_1 \times \cdots \times n_d}$ through eq. \eqref{TT_definition} of the main paper:
\begin{center}
    $M_{TT}(W_1, W_2, \ldots, W_d)$ is a tensor of size $(n_1, \ldots, n_d)$ with value $W_1(i_1)W_2(i_2) \cdots W_d(i_d)$ at index $(i_1, \ldots, i_d)$
\end{center}
We denote the input space $L_{\mathrm{cores}}$ and the resulting map $M_{TT} \colon L_{\mathrm{cores}} \to \reals^{n_1 \times \cdots \times n_d}$. Note that $M_{TT}$ is surjective to $\Mr$ but not injective. Then the curve $c_{TT}(t) := M_{TT}(c_1(t), c_2(t), \ldots, c_d(t))$, where $c_k(t)$ is a smooth curve on $\reals^{r_{k-1} \times n_k \times r_k}$, is a smooth curve in $\reals^{n_1 \times \cdots \times n_d}$. 

Finally, we can write eq. \eqref{differential_original} as the following:
\begin{align}\label{differential_split}
	(\D_V \P_X) Z = \sum_{i=1}^d\lim_{t \to 0} \frac{\P_{c_{TT}(t)}^i(Z) - \P_{c_{TT}(0)}^i(Z)}{t}
\end{align}
where $c_{TT}(t)$ is defined as above, $c_{TT}(0) = X$, $c_{TT}'(0) = V$, and for every fixed $t$ in some small neighborhood $(-\epsilon, \epsilon)$, $c_1(t), \ldots, c_d(t)$ is a left-orthogonal, minimal TT-decomposition and $c_{TT}(t) \in \Mr$. We now establish the existence of such a curve $c_{TT}(t)$:
\begin{lemma}\label{core_curve_lemma}
    There exists a smooth curve $c_{TT}(t)$ on $\Mr$ such that $c_{TT}(0) = X$, $c_{TT}'(0) = V$, and for every fixed $t \in (-\epsilon, \epsilon)$, where $\epsilon > 0$ is fixed, we have that $c_1(t), \ldots, c_d(t)$ is a minimal and left-orthogonal TT-decomposition, where $c_{TT}(t)$ is of the following form:
    \begin{align}
        c_{TT}(t) = M_{TT}(U_1 + t\delta V_1 + O(t^2), \ldots, U_{d-1} + t\delta V_{d-1} + O(t^2), U_d + t\delta V_d).
    \end{align}
\end{lemma}

Before starting the proof for Lemma \ref{core_curve_lemma}, it is important to note that we represent $V$ by the original parametrization $\delta V_1, \ldots, \delta V_d$, i.e. from the formula 
\begin{align}\label{first_param}
    V(i_1, \ldots, i_d) &= \sum_{k=1}^d U_1(i_1) \cdots U_{k-1}(i_{k-1}) \delta V_k(i_k) U_{k+1}(i_{k+1}) \cdots U_d(i_d).
\end{align}
This is in contrast to the alternative parametrization $\delta \widetilde{V}_1, \ldots, \delta \widetilde{V}_d$ from which we get the formulas \eqref{proj_explicit} and \eqref{proj_d} for the orthogonal projector. We use the original parametrization because it is more intuitive to construct a curve $c_{TT}(t)$ such that $c_{TT}'(0) = V$ using $\delta V_1, \ldots, \delta V_d$ due to the resemblance of eq. \eqref{first_param} to the product rule. We now prove a note from Section \ref{alternative} that we are able to easily interchange between these two parametrizations:

Thus, the original parametrization $\delta V_1, \ldots, \delta V_d$ is always accessible from the parametrization $\delta \widetilde{V}_1, \ldots, \delta \widetilde{V}_d$. We now move on to prove Lemma \ref{core_curve_lemma}. 
\begin{proof}
    First, we will construct a curve $c_{TT}(t)$ that satisfies all requirements except for $c_1(t), \ldots, c_d(t)$ to be minimal. Then, we will simply choose $\epsilon$ small enough such that $c_1(t), \ldots, c_d(t)$ is also minimal for $t \in (-\epsilon, \epsilon)$.

    We start by constructing a curve satisfying left-orthogonality. Left-orthogonality requires that $(c_k(t))^L$ has orthonormal columns for all $k \in [d-1]$. Note here that we can use the structure of a different manifold, the \textit{Stiefel manifold}:
    \begin{align*}
        \mathrm{St}(n,p) = \left \{ M \in \reals^{n \times p} : M^\top M = I_p \right\}
    \end{align*}
    where the tangent space $\T_M \mathrm{St}(n,p)$ at a point $M$ is the set of all matrices $A$ such that $M^\top A + A^\top M = 0$. More details on the Stiefel manifold can be found in \cite[\S 7.3]{boumal2020intromanifolds}. Note that for every $k \in [d-1]$, $U_k^L \in \mathrm{St}(r_{k-1} n_k, r_k)$ by left-orthogonality. Furthermore, $(\delta V_k^L)^\top U_k^L = 0$ (called the \textit{gauge conditions}) by definition of the parametrization $\delta V_1, \ldots, \delta V_d$ of $V$, and thus $(\delta V_k^L)^\top U_k^L + (U_k^L)^\top \delta V_k^L = 0 \implies \delta V_k^L \in \T_{U_k^L} \mathrm{St}(r_{k-1} n_k, r_k)$. This then allows us to construct a smooth curve $(c_k(t))^L = U_k^L + t\delta V_k^L + O(t^2)$ such that, for all $t$ in some small neighborhood $(-\epsilon_0, \epsilon_0)$, we have that $(c_k(t))^L \in \mathrm{St}(r_{k-1} n_k, r_k)$. Construct $c_k(t)$ by unfolding $(c_k(t))^L$ back into a tensor of order 3 so that $c_k(t) = U_k + t\delta V_k + O(t^2)$, repeat this construction for all $k \in [d-1]$, and finally constuct $c_d(t) = U_d + t \delta V_d$.

    It is easy to verify that the resulting curve $c_{TT}(t) = M_{TT}(U_1 + t\delta V_1 + O(t^2), \ldots, U_{d-1} + t\delta V_{d-1} + O(t^2), U_d + t\delta V_d)$ satisfies $c_{TT}(0) = X$ and $c_{TT}'(0) = V$, since $M_{TT}$ is a multilinear map. We are then left to prove that there exists a neighborhood $(-\epsilon, \epsilon)$ small enough such that $c_1(t), \ldots, c_d(t)$ is minimal for all $(-\epsilon, \epsilon)$. There is a statement in \citep[\S 9.3.3]{uschmajew2020geometrylowrank} stating that the decomposition $c_1(t), \ldots, c_d(t)$ is minimal if and only if $(c_k(t))^L$ and $(c_k(t))^R$ are both of full rank for all $k \in [d]$. This statement can be proven using Lemma \ref{TT} and rank arguments, and it is useful because the set of full rank matrices is an open set. Therefore, for each core $U_k$, we can construct an open ball around $U_k$ in the space $\reals^{r_{k-1} \times n_k \times r_k}$ of some positive radius $\epsilon_L^k$, denoted $\mathcal{N}_{\epsilon_L^k}(U_k)$, such that all tensors in $\mathcal{N}_{\epsilon_L^k}(U_k)$ have left-flattenings of full rank. Similarly, we can construct an open ball $\mathcal{N}_{\epsilon_R^k}(U_k)$ for each core $U_k$ such that all tensors in $\mathcal{N}_{\epsilon_R^k}(U_k)$ have right-flattenings of full rank. Defining $\epsilon_k := \min(\epsilon_L^k, \epsilon_R^k)$, it then follows that if a TT-tensor $W$ has cores $W_1, \ldots, W_d$ such that $W_k \in \mathcal{N}_{\epsilon^k}(U_k)$ for all $k \in [d]$, then $W_1, \ldots, W_d$ is a minimal decomposition, and $W \in \Mr$. Since each $c_k(t)$ is a smooth curve, there exists $\delta_k > 0$ such that $t \in (-\delta_k, \delta_k) \implies c_k(t) \in \mathcal{N}_{\epsilon^k}(U_k)$. Setting $\epsilon := \min (\delta_1, \ldots, \delta_d, \epsilon_0)$ finishes the proof. 
\end{proof}

Now that we have constructed $c_{TT}$ to be a valid curve, the following lemma comes trivially:

\begin{lemma} \label{standard_diff}
    Let $X \in \Mr$ and $V \in \T_X \Mr$. The following identity holds:
  \begin{align}
        (\D_V \P_X^i) Z = \lim_{t \to 0} \frac{\P^{i, Z}(c_{TT}(t)) - \P^{i, Z}(c_{TT}(0))}{t}
  \end{align}
\end{lemma}

\subsection{Proofs for the correction term of the Riemannian Hessian}

We can now find a formula for $(\D_V \P_X^i) Z$. Recalling equations \eqref{proj_explicit} and \eqref{proj_d}, we can evaluate $(\D_V \P_X^i) Z$ by evaluating the derivatives of the terms $X_{\le k}X_{\le k}^\top$ and $\Xt_{>k} \Xt_{>k}^\top$. We then write $\D_V (X_{\le k}X_{\le k}^\top)$ to mean the following:
\begin{equation*}
    \D_V (X_{\le k}X_{\le k}^\top) := \lim_{t \to 0} \frac{X_{\le k}(c_{TT}(t)) X_{\le k}^\top(c_{TT}(t)) - X_{\le k}(c_{TT}(0)) X_{\le k}^\top(c_{TT}(0))}{t}
\end{equation*}
where $X_{\le k}(c_{TT}(t))$ is the $k$th left interface matrix using cores from $c_{TT}(t)$. We now build up derivations for $D_V (X_{\le k}X_{\le k}^\top)$ and $D_V (X_{\le k}X_{\le k}^\top)$.

\begin{lemma}\label{interface_derivative}
  Let $X \in \Mr$ and $V \in \T_X \Mr$. The following identities hold for all $k \in [d]$:
  \begin{enumerate} 
    \item $\D_V (X_{\le k}) = V_{\le k}$ \label{interface_derivative_left}
    \item $\D_V (X_{\ge k}) = V_{\ge k}$ \label{interface_derivative_right}
    \end{enumerate}
\end{lemma}
\begin{proof}
  We will only write a proof for part \ref{interface_derivative_left}, as the proof for part \ref{interface_derivative_left} is analogous.

  We proceed via induction. Note that by construction of $c_{TT}(t)$, it follows that for all $k \in [d]$, $\D_V(U_k^L) = \delta V_k^L$. This covers the base case, since $\D_V (X_{\le 1}) = \D_V(U_1^L) = \delta V_1^L = V_{\le 1}$. Now, assuming the statement holds for $k-1$, we conclude that $\D_V(I_{n_k} \otimes X_{\le k-1}) = I_{n_k} \otimes V_{\le k-1}$ and that:

  \begin{align*}
    \D_V (X_{\le k}) &= \D_V \left((I_{n_k} \otimes X_{\le k-1})U_k^L \right) \\
    &= \D_V(I_{n_k} \otimes X_{\le k-1})U_k^L + (I_{n_k} \otimes X_{\le k-1})\D_V(U_k^L) \\
    &= (I_{n_k} \otimes V_{\le k-1})U_k^L + (I_{n_k} \otimes X_{\le k-1}) \delta V_k^L = V_{\le k}.
  \end{align*}
\end{proof}

Using Lemma \ref{interface_derivative}, we can then evaluate $\D_V (X_{\le k}X_{\le k}^\top)$ through a standard application of the product rule.
\begin{lemma} \label{eq_left_variational}
    The following identity holds:
    \begin{equation}
        \D_V (X_{\le k}X_{\le k}^\top) = V_{\le k}X_{\le k}^\top + X_{\le k}V_{\le k}^\top
    \end{equation}
\end{lemma}

Next, we evaluate $\D_V (\Xt_{> k}\Xt_{> k}^\top)$. For this differential, we need the following useful cancellation identity:
Furthermore, the left variational interface matrices $V_{\le k}$ in particular have an important cancellation identity.
\begin{lemma}\label{variational_cancellation}
    For all $k \in [d-1]$, the following cancellation identity holds: $V_{\le k}^\top X_{\le k} = 0$.
\end{lemma}
\begin{proof}
    We prove by induction. The base case $k=1$ is given directly by the gauge conditions, since $V_{\le 1}^\top X_{\le 1} = (\delta V_1^L)^\top U_1^L = 0$. Assuming the desired cancellation holds for $k-1$, we expand $V_{\le k}^\top X_{\le k}$ inductively via definitions to get:
    \begin{align*}
        &\quad \left( (I_{n_k} \otimes V_{\le k-1})U_k^L + (I_{n_k} \otimes X_{\le k-1}){\delta V}_k^L \right)^\top \left( (I_{n_k} \otimes X_{\le k-1})U_k^L \right) \\
        &= (U_k^L)^\top(I_{n_k} \otimes V_{\le k-1}^\top X_{\le k-1})U_k^L + ({\delta V}_k^L)^\top(I_{n_k} \otimes X_{\le k-1}^\top X_{\le k-1})U_k^L \\
        &= 0 + ({\delta V}_k^L)^\top U_k^L = 0
    \end{align*}
\end{proof}
\begin{lemma} \label{eq_right_variational}
The following equation holds:
  \begin{align*}
    \D_V (\Xt_{\ge k+1}\Xt_{\ge k+1}^\top) &= (I_{n_{k+1}\cdots n_d} - \Xt_{\ge k+1} \Xt_{\ge k+1}^\top)V_{\ge k+1} R_k^{-1} \Xt_{\ge k+1}^\top \\
    &\quad + \Xt_{\ge k+1} R_k^{-\top} V_{\ge k+1}^\top(I_{n_{k+1}\cdots n_d} - \Xt_{\ge k+1} \Xt_{\ge k+1}^\top)
  \end{align*}
\end{lemma}

\begin{proof}
  We will first introduce an index-free shorthand notation for the sake of readability; this is used for all subsequent proofs in this paper. For example, $R := R_k$, $X_> := X_{> k}$, $X_\ge := X_{\ge k}$, $X_> := X_{\ge k+1}$, and the same notation for the left interface matrices. All identity matrices are abbreviated to $I$ with size implied from context, e.g. $(I \otimes X_<) = (I_{n_k} \otimes X_{\le k-1})$. 
 
  We start the proof by differentiating both sides of the defining equality for $R_k$, namely $X_> = \Xt_> R$
  \begin{align*}
    \D_V(X_>) &= \D_V(\Xt_> R) \\
    V_> &= \D_V(\Xt) R + \Xt_> \D_V(R) \\
    V_> &= (\Xt_> \Omega + \Xt_>^\perp B) R + \Xt_> \D_V(R)
  \end{align*}
  where $\Omega \in \reals^{r_k \times r_k}$ is skew-symmetric, and $B \in \reals^{n_{k+1} \cdots n_d - r_k \times r_k}$. This decomposition comes from the fact that $\widetilde{W}_>(X_t^V)$ is a smooth curve on the Stiefel manifold, so we can use the structure of its tangent space to generate $\Omega, B$. We want to find more explicit structure for the two unknowns we artificially introduced: $\Omega$ and $B$. Starting with $B$, we multiply both sides by $(\Xt_>^\perp)^\top$, where $\Xt_>^\perp$ is the orthonormal complement of $\Xt_>$, on the left to get the following expression for $B$:

  \begin{align*}
    (\Xt_>^\perp)^\top V_> &= BR \\
    B &= (\Xt_>^\perp)^\top V_> R^{-1}
  \end{align*}

  We then proceed to evaluate the desired equality:

  \begin{align*}
      &= (\Xt_> \Omega + \Xt_>^\perp B) \Xt_>^\top + \Xt_> (\Xt_> \Omega + \Xt_>^\perp B)^\top \\
      &= (\Xt_> \Omega + \Xt_>^\perp (\Xt_>^\perp)^\top V_> R^{-1})\Xt_>^\top + \Xt_>(\Xt_> \Omega + \Xt_>^\perp (\Xt_>^\perp)^\top V_> R^{-1})^\top \\
      &= (\Xt_> \Omega + \Xt_>^\perp (\Xt_>^\perp)^\top V_> R^{-1})\Xt_>^\top + \Xt_>(\Omega^\top \Xt_>^\top + R^{-T} V_>^\top \Xt_>^\perp (\Xt_>^\perp)^\top ) \\
      &= \Xt_>^\perp (\Xt_>^\perp)^\top V_> R^{-1} \Xt_>^\top + \Xt_> R^{-T} V_>^\top\Xt_>^\perp (\Xt_>^\perp)^\top + \Xt_> \Omega \Xt_>^\top + \Xt_> \Omega^\top \Xt_>^\top \\
      &= (I - \Xt_> \Xt_>^\top)V_> R^{-1} \Xt_>^\top + \Xt_> R^{-\top} V_>^\top(I - \Xt_> \Xt_>^\top) + \Xt_> (\Omega + \Omega^\top) \Xt_>^\top \\
      &= (I - \Xt_> \Xt_>^\top)V_> R^{-1} \Xt_>^\top + \Xt_> R^{-\top} V_>^\top(I - \Xt_> \Xt_>^\top)
  \end{align*}
\end{proof}

Using Lemmas \ref{eq_left_variational} and \ref{eq_right_variational}, we can then evaluate the differential $\D_V \P_X^k Z$ through a simple application of the product rule:

\begin{lemma}\label{explicit_differential}
    The differential for $\P_i Z$ is given by:
    \begin{align*}
        (\De_V \P_X^k Z)^{<k>} &= \left(I_{n_k} \otimes (V_{< k} X_{< k}^\top + X_{< k} V_{< k}^\top) - (V_{\le k} X_{\le k}^\top + X_{\le k} V_{\le k}^\top)\right)Z^{<k>} \Xt_{> k} \Xt_{> k}^\top \\
        &+ (I_{n_k} \otimes X_{< k}X_{< k}^\top - X_{\le k} X_{\le k}^\top)Z^{<k>}(I_{n_{k+1}\cdots n_d} - \Xt_{> k} \Xt_{> k}^\top) V_{> k} R_i^{-1} \Xt_{> k}^\top \\
        &+(I_{n_k} \otimes X_{< k}X_{< k}^\top - X_{\le k} X_{\le k}^\top)Z^{<k>}\Xt_{> k} R_k^{-T} V_{> k}^\top (I_{n_{k+1}\cdots n_d} - \Xt_{> k} \Xt_{> k}^\top),
    \end{align*}

    for $i < d$, and 

    \begin{equation*}
        (\De_V \P_X^d Z)^{<d>} = (I \otimes (V_{< d} X_{< d}^\top + X_{< d} V_{< d}^\top))Z^{<d>}.
    \end{equation*}
\end{lemma}

\noindent We can then directly apply Lemma \ref{explicit_differential} to prove Theorem \ref{TT_diagonal_terms} of the main paper.
\vspace{4mm}
\begin{proof}[Proof of Theorem \ref{TT_diagonal_terms}]
    Following the convention for the shorthand introduced above. For the $k < d$ case:
    \begin{align*}
        &\quad (\P_X^k \D_V \P_X^k Z)^{<k>} \\
        &=  (\P_X^k \left((I\otimes (V_{< } X_{< }^\top + X_{< } V_{< }^\top) - (V_{\le } X_{\le }^\top + X_{\le } V_{\le }^\top))Z^{<k>} \Xt_{> } \Xt_{> }^\top\right))^{<k>} \stepcounter{equation}\tag{\theequation}\label{summand1}\\
        &\quad + (\P_X^k\left((I \otimes X_{< }X_{< }^\top - X_{\le } X_{\le }^\top)Z^{<k>}(I - \Xt_{> } \Xt_{> }^\top) V_{> } R_i^{-1} \Xt_{> }^\top\right))^{<k>} \stepcounter{equation}\tag{\theequation}\label{summand2} \\
        &\quad + (\P_X^k((I \otimes X_{< }X_{< }^\top - X_{\le } X_{\le }^\top) Z^{<k>}\Xt_{> } R^{-T} V_{> }^\top (I - \Xt_{> } \Xt_{> }^\top))^{<k>}.\stepcounter{equation}\tag{\theequation}\label{summand3}
    \end{align*}
    Recall the formula for $(\P_X^k Z)^{<k>}$ given in eq. \eqref{proj_explicit}. The operation consists of left multiplication by $(I_{n_k} \otimes (X_{<} X_{<}^\top) - X_{\le } X_{\le }^\top$ and right multiplication by $\Xt_{> } \Xt_{> }^\top$. Since $(I_ - \Xt_{> k} \Xt_{> k}^\top)\Xt_{> k} \Xt_{> k}^\top = \Xt_{> } \Xt_{> }^\top - \Xt_{> } \Xt_{> }^\top = 0$, the 3rd summand \eqref{summand3} vanishes. Since $(I \otimes X_{< }X_{< }^\top - X_{\le } X_{\le }^\top)^2 = (I_ \otimes X_{< }X_{< }^\top - X_{\le } X_{\le }^\top)$, \eqref{summand2} remains unchanged from $\P_X^k$. We can simplify \eqref{summand1} in the following way:
    \begin{align*}
        &(\P_X^k \left((I \otimes (V_{< } X_{< }^\top + X_{< } V_{< }^\top) - (V_{\le } X_{\le }^\top + X_{\le } V_{\le }^\top))Z^{<k>} \Xt_{> } \Xt_{> }^\top\right))^{<k>}\\
        &= (I \otimes X_{< }X_{< }^\top - X_{\le } X_{\le }^\top)(I \otimes (V_{< } X_{< }^\top + X_{< } V_{< }^\top) - (V_{\le } X_{\le }^\top + X_{\le } V_{\le }^\top)) Z^{<k>} \Xt_{> } \Xt_{> }^\top \\
        &= \left(I \otimes X_< V_<^\top - (I \otimes X_<)U_k^L V_\le ^\top - (I \otimes X_<)\delta V_k^L X_\le^\top - X_\le (U_k^L)^\top(I \otimes V_<^\top) + X_\le V_\le^\top  \right) \\
        &\quad \cdot Z^{<k>} \Xt_{> } \Xt_{> }^\top \\
        &= \left(I \otimes X_< V_<^\top - X_\le V_\le ^\top - (I \otimes X_<)\delta V_k^L X_\le^\top - X_\le (U_k^L)^\top(I \otimes V_<^\top) + X_\le V_\le^\top  \right) Z^{<k>} \Xt_{> } \Xt_{> }^\top \\
        &= \left(I \otimes X_< V_<^\top- (I \otimes X_<)\delta V_k^L (U_k^L)^\top (I \otimes X_<^\top) - (I \otimes X_<) U_k^L (U_k^L)^\top(I \otimes V_<^\top) \right) Z^{<k>} \Xt_{> } \Xt_{> }^\top .
    \end{align*}
    Factoring out $(I \otimes X_<)$ from \eqref{summand2} and adding it to the result above gets the desired result:
    \begin{align*}
        &= (I \otimes X_{< }) \left(((I - U_k^L(U_k^L)^\top)(I \otimes V_{< }^\top) )Z^{<k>} \Xt_{> } \right) \Xt_{> }^\top \\
        &\quad - (I \otimes X_{< }) \left( \delta V_k^L (U_k^L)^\top (I \otimes X_{< }))Z^{<k>} \Xt_{> } \right) \Xt_{> }^\top \\
        &\quad + (I \otimes X_{< })\left( (I - U_k^L(U_k^L)^\top)(I \otimes X_{<}^\top)Z^{<k>}(I - \Xt_{> } \Xt_{> }^\top) V_> R^{-1}\right)\Xt_{>}^\top.
    \end{align*}
    The $k = d$ case follows much more straightforwardly:
    \begin{align*}
        (\P_X^d \D_V \P_X^d Z)^{<d>} &= \P_X^d (I \otimes (V_{< d} X_{< d}^\top + X_{< d} V_{< d}^\top))Z^{<d>} \\
        &= (I \otimes (X_{< d} X_{< d}^\top))(I \otimes (V_{< d} X_{< d}^\top + X_{< d} V_{< d}^\top))Z^{<d>}\\
        &= (I \otimes (X_{< d} V_{< d}^\top))Z^{<d>}\\
        &= (I \otimes X_{< d})(I \otimes V_{< d}^\top)Z^{<d>}
    \end{align*}
\end{proof}

Recall that in the main paper we introduced the formula $\P_X^i \D_V \P_X^j Z = -(\De_V \P_X^i) \P_X^j Z$ as a key component for deriving the cross-term formulae. We provide here a proof of this statement: 
\begin{lemma}\label{cross_term_exchange}
    For all $i, j \in [d], i \ne j$, the following identity holds: $\P_X^i \D_V \P_X^j Z = -(\De_V \P_X^i) \P_X^j Z$
\end{lemma}
\vspace{4mm}
\begin{proof}
  First, note that $\P_X^i \P_X^j Z = 0$ for all $X \in \Mr$ by orthogonality. From this, we directly get that $\D_V(\P_X^i \P_X^j Z) = 0$. From Lemma \ref{explicit_differential}, we know that differentials of the form $\D_V \P_X^k Z$ exist for all fixed tensors $Z$, allowing us to split $\D_V(\P_X^i \P_X^j Z)$ by the product rule:
  \begin{align*}
    \D_V(\P_X^i \P_X^j Z) &= 0 \\
    \lim_{t \to 0} \frac{\P_{c(t)}^i \P_{c(t)}^j Z - \P_{c(0)}^i \P_{c(0)}^j Z}{t} &= 0 \\
    \lim_{t \to 0} \frac{\P_{c(t)}^i \P_{c(t)}^j Z - \P_{c(0)}^i \P_{c(t)}^j Z + \P_{c(0)}^i \P_{c(t)}^j Z - \P_{c(0)}^i \P_{c(0)}^j Z}{t} &= 0 \\
    \lim_{t \to 0} \frac{(\P_{c(t)}^i - \P_{c(0)}^i )\P_{c(t)}^j Z}{t} + \lim_{t \to 0} \frac{\P_{c(0)}^i (\P_{c(t)}^j Z - \P_{c(0)}^j) Z}{t} &= 0 \\
    (\D_V \P_X^i) \P_X^j Z + \P_X^i \D_V \P_X^j Z &= 0 \\
    \P_X^i \D_V \P_X^j Z &= -(\D_V \P_X^i) \P_X^j Z
  \end{align*}
\end{proof}

Finally, we will prove Theorem \ref{TT_cross_terms} of the main paper:

\begin{proof}[Proof for Theorem \ref{TT_cross_terms}]
  Denote $Y_\le^j$ and $Y_>^j$ to be interface matrices for a TT-decomposition $U_1, \ldots U_{j-1}, \delta Y_j, \widetilde{U}_{j+1}, \ldots, \widetilde{U}_d$, where $\delta Y_j$ are the variational cores for the tangent vector $\P_X^j(Z)$. Starting with the $j > i$ case, we can use Theorem 3.6 to write the following:

  \begin{align*}
    (\P_i \D_V \P_j Z)^{<i>} &= (-(\D_V \P_i) \P_j Z)^{<i>}  \\
    &= -\left(I \otimes (V_{< i} X_{< i}^\top + X_{< i} V_{< i}^\top) - (V_{\le i} X_{\le i}^\top + X_{\le i} V_{\le i}^\top)\right)X_{\le i} (Y^j_{> i})^\top \Xt_{> i} \Xt_{> i}^\top \\
      &\quad \quad - (I \otimes X_<X_<^\top - X_\le X_\le^\top)X_\le (Y^j_>)^\top(I - \Xt_> \Xt_>^\top) V_> R^{-1} \Xt_>^\top \\
      &\quad \quad -(I \otimes X_<X_<^\top - X_\le X_\le^\top)X_\le (Y^j_>)^\top\Xt_> R^{-T} V_>^\top (I - \Xt_> \Xt_>^\top)
  \end{align*}

  Along with the identity $V_{\le k}^\top X_{\le k} = 0$, it holds that $(I \otimes X_<X_<^\top - X_\le X_\le^\top)X_\le = 0$:
  \begin{align*}
    (I \otimes X_<X_<^\top - X_\le X_\le^\top)X_\le &= (I \otimes X_<X_<^\top)X_\le - X_\le X_\le^\top X_\le \\
    &= (I \otimes X_<X_<^\top)(I \otimes X_<)U^L - X_\le \\
    &= (I \otimes X_<X_<^\top X_<)U^L - X_\le \\
    &= (I \otimes X_<)U^L - X_\le \\
    &= X_\le - X_\le = 0
  \end{align*}
  Thus, we are left only with the first summand:

  \begin{align*}
      (-(\D_V \P_i) \P_j Z)^{<i>} &= -(I \otimes (V_< X_<^\top) - V_\le X_\le^\top)X_\le (Y^j_>)^\top \Xt_> \Xt_>^\top \\
      &= \left(V_\le - (I \otimes V_<) U^L \right)(Y^j_>)^\top \Xt_> \Xt_>^\top \\
      &= (I \otimes X_<) \delta U_i^L (Y_>^j)^\top \Xt_> \Xt_>^\top
  \end{align*}

  Giving us the desired form for the $j > i$ case. For the $j < i$ case, we have the following similar expression:

  \begin{align*}
      (-(\D_V \P_i) \P_j Z)^{<i>} &= -\left(I \otimes (V_< X_<^\top + X_< V_<^\top) - (V_\le X_\le^\top + X_\le V_\le^\top)\right)Y^j_\le \Xt_>^\top \Xt_> \Xt_>^\top \\
      &- (I \otimes X_<X_<^\top - X_\le X_\le^\top)Y^j_\le \Xt_>^\top(I - \Xt_> \Xt_>^\top) V_> R^{-1} \Xt_>^\top \\
      &-(I \otimes X_<X_<^\top - X_\le X_\le^\top)Y^j_\le \Xt_>^\top\Xt_> R^{-T} V_>^\top (I - \Xt_> \Xt_>^\top) \\
  \end{align*}

  Note that both $X_{\le k}^\top Y^j_{\le k} = 0$ and $X_{< k}^\top Y^j_{< k} = 0$; this is because $Y^j_{\le k}$ and $Y^j_{< k}$ are simply variational interface matrices for a tangent vector with cores $\{0, \ldots, 0, \delta Y_j, 0, \ldots, 0\}$, allowing us to use Lemma \ref{variational_cancellation}. It follows that both $X_{\le }^\top Y^j_{\le } = 0$ and $(I \otimes X_<^\top) Y^j_\le = (I \otimes X_<^\top Y^j_<)U^L = 0$, so the second and third summands are indeed 0. We are then left with the first summand:

  \begin{align*}
      (-(\D_V \P_i) \P_j Z)^{<i>} &= -\left(I \otimes (X_< V_<^\top) - (X_\le V_\le^\top)\right)Y^j_\le \Xt_>^\top \Xt_> \Xt_>^\top \\
      &= -(I \otimes X_< V_<^\top - X_\le V_\le^\top)Y^j_\le \Xt_>^\top
  \end{align*}

  This finishes the $j < i$ case. The 3rd case, $(\P_X^d \D_V \P_X^j Z)^{<d>}$, is trivial.

\end{proof}

\section{Computational complexity analysis}\label{proofs_complexity}

In this section, we will outline an algorithm to compute the correction term on $\Mr$ given formulas from Theorems \ref{TT_diagonal_terms} and \ref{TT_cross_terms}. In general, we want to avoid the ``curse of dimensionality'' by avoiding any computation on the order of $O(n^d)$. For simplicity, we restrict the rank and size vectors in this section to be uniform, i.e. $\mathbf{n} = (n, \ldots, n)$ and $\mathbf{r} = (r, \ldots r)$.

To begin this section, we establish a number of computations that we know can be done in feasible time:

\begin{corollary}\label{computation_init}
    The following can be computed in $O(dnr^3)$ flops:

    \begin{enumerate}
        \item \label{comp_init_1} Given a TT-tensor $X$ with left-orthogonal decomposition $\{U_1, \ldots, U_d\}$, the right-orthogonal decomposition $\{\widetilde{U}_1, \ldots, \widetilde{U}_d\}$ along with matrices $\{R_1,\ldots , R_{d-1}\}$ as defined in Section \ref{orthogonal_decomp} are computable in $O(dnr^3)$ flops.
        \item \label{comp_init_2} Given two TT-tensors $X$ and $Y$ of size $\mathbf{n}$, along with their TT-decompositions of size $\mathbf{r}$, the set of matrices $\{X_{\le 1}^\top Y_{\le 1}, \ldots, X_{\le d}^\top Y_{\le d}\}$ is computable in $O(dnr^3)$ flops. Similarly, the set of matrices $\{X_{\ge 1}^\top Y_{\ge 1}, \ldots, X_{\ge d}^\top Y_{\ge d}\}$ is also computable in $O(dnr^3)$ flops.
    \end{enumerate}
\end{corollary}

This is a corollary from \cite[Alg. 4.1]{steinlechner2016riemannian} and \cite[Alg. 4.2]{steinlechner2016riemannian} respectively. Note that \cite[Alg. 4.2]{steinlechner2016riemannian} is presented as an algorithm to compute the inner product; however, the algorithm achievies this by computing $X_{\le d}^\top Y_{\le d} = \inner{X}{Y}$, and the intermediate steps for computing $X_{\le d}^\top Y_{\le d}$ are indeed $\{X_{\le 1}^\top Y_{\le 1}, \ldots , X_{\le d-1}^\top Y_{\le d-1} \}$. There is analogous procedure stated in \cite[\S 4.2.3]{steinlechner2016riemannian} that computes $X_{\ge 1}^\top Y_{\ge 1}$.

\subsection{Computational complexity for the diagonal terms}

We will first show the general $O(dz + dnr^3)$ estimate, then show the $O(z + dnr^3)$ estimate for tensor completion and argue why we would see this estimate for typical applications. We assume that we are given a left-orthogonal decomposition $\{U_1, \ldots, U_d\}$ for TT-tensor $X$, as well as tangent cores $\{\delta\widetilde{V}_1, \ldots, \delta\widetilde{V}_d\}$ for tangent vector $V$.

For all cases, the limiting computations will be to compute the following: 
\begin{align}
    &\{(I \otimes X_{< k})Z^{<k>}\Xt_{> k}\}_{k=1}^d,\label{comp_essential_1}\\
    &\{(I \otimes V_{< k})Z^{<k>}X_{> k}\}_{k=1}^d\text{, and}\label{comp_essential_2} \\
    &\{(I \otimes X_{< k})Z^{<k>}V_{> k}\}_{k=1}^d.\label{comp_essential_3}
\end{align}
After computing these matrices, we can build up the computation for the core of the diagonal term $\P_X^k \D_V \P_X^k Z$ through small matrix multiplications ($U_k^L, \delta V_k^L \in \reals^{rn \times r}$, $R_k \in \reals^{r \times r}$). Obtaining all of the $R_k$ matrices can be done in $O(dnr^3)$ flops by Corollary \ref{computation_init}.\ref{comp_init_2}, and since $R_k$ is upper triangular, each $R_k^{-1}$ is computable in under $O(nr^3)$ flops. Note we need to extract the $\delta V_k$ parametrization from the $\delta \widetilde{V}_k$ parametrization; this can be done under $O(dnr^3)$ flops total by Lemma \ref{interchange_tangent_param}. Finally, we are left with the computation of $\Xt_{> k}^\top V_{> k}$; since each matrix is in $\reals^{n_{k+1} \cdots n_d \times r_k}$, computing the products explicity would yield a curse of dimensionality. Note that this product is similar to the one given in Corollary \ref{computation_init}.\ref{comp_init_2}, which is an order $O(dnr^3)$ computation. Indeed, we are able to make a similar algorithm to compute all such $\Xt_{> k}^\top V_{> k}$ products in a total of $O(dnr^3)$ flops. After finishing small matrix multiplications, this completes the computation of the variational cores for all diagonal terms $\P_X^k \D_V \P_X^k Z$ in an extra $O(dnr^3)$ flops, on top of the computation of \eqref{comp_essential_1}\eqref{comp_essential_2}\eqref{comp_essential_3}. This then leaves us to estimate the computation for \eqref{comp_essential_1}\eqref{comp_essential_2}\eqref{comp_essential_3}.

\begin{algorithm}\label{alg1}
    \SetAlgoLined
    \KwResult{The set of matrices $\{\Xt_{> 1}^\top V_{> 1}, \ldots, \Xt_{> d-1}^\top V_{> d-1} \}$ }
    \textbf{Input: } TT-cores $\{U_1, \ldots U_d\}$, $\{\widetilde{U}_1, \ldots \widetilde{U}_d\}$, and variational cores $\{\delta V_1, \ldots, \delta V_d\}$ \;

    $p_1 \leftarrow \widetilde{U}_d^R(\delta V_d^R)^\top$\;

    $p_2 \leftarrow \widetilde{U}_d^R(U_d^R)^\top$\;

    $\Xt_{> d-1}^\top V_{> d-1} \leftarrow p_1$\;

     \For{$k = d-1,\ldots, 2$}{
      $A \leftarrow \widetilde{U}_k \times_3 p_1^\top$ \textit{($A$ is a tensor of size $(r_{k-1}, n_k, r_k)$)}\;

      $B \leftarrow \widetilde{U}_k \times_3 p_2^\top$ \;

      $p_1 \leftarrow A^R(U_k^R)^\top + B^R (\delta V_k^R)^\top$ \textit{(note the resemblance of Definition \ref{variational_interface})}\;

      $\Xt_{> k-1}^\top V_{> k-1} \leftarrow p_1$\;

      $p_2 \leftarrow B^R (U_k^R)^\top$\;

     }
     \caption{Computing $\{\Xt_{> k}^\top V_{> k}\}_{k=1}^d$ (adapted from \citep[Alg. 4.2]{steinlechner2016riemannian})}
\end{algorithm}

In typical algorithms, there is some structure on the Euclidean gradient $Z$ that allows for more efficient computation of \eqref{comp_essential_1}\eqref{comp_essential_2}\eqref{comp_essential_3}. For example, in tensor completion, the gradient $Z$ is sparse, and we can compute \eqref{comp_essential_1} in $O(d|\Omega|r^2)$ flops using \citep[Alg. 5.2]{steinlechner2016riemannian}, where $\Omega$ is the set of observed indices.

As can be seen from the tensor completion example, tight estimates for the computational complexity of \eqref{comp_essential_1}\eqref{comp_essential_2}\eqref{comp_essential_3} will be dependent on the gradient's structure. Nonetheless, we argue that in typical algorithms, computing \eqref{comp_essential_1}\eqref{comp_essential_2}\eqref{comp_essential_3} will typically take about as much time as computing \eqref{comp_essential_1}, and computing the diagonal term sum $\sum_{k=1}^d\P_X^k \D_V \P_X^k (Z)$ will take about as much time as computing $\P_X(Z)$.

To illustrate this claim, we present pseudocode for an algorithm for an algorithm for efficient computing \eqref{comp_essential_1}\eqref{comp_essential_2}\eqref{comp_essential_3} when the gradient $Z$ is sparse. This algorithm adapts heavily from \citep[Alg. 5.2]{steinlechner2016riemannian}, the efficient algorithm for computing \eqref{comp_essential_1} for sparse gradients. This algorithm highlights the general approach to transform algorithms that compute \eqref{comp_essential_1} into those that compute \eqref{comp_essential_1}\eqref{comp_essential_2}\eqref{comp_essential_3}: alongside the original computations of matrices $(I \otimes X_{< k})Z^{<k>}\Xt_{> k}$ that utilize Def. \ref{def:interface}, we can build up $(I \otimes V_{< k})Z^{<k>}\Xt_{> k}$ and $(I \otimes X_{< k})Z^{<k>}V_{> k}$ in the same manner that instead utilize Def. \ref{variational_interface}.

\begin{algorithm}
    \label{alg2}
    \SetAlgoLined
    \KwResult{Three set of $d$ matrices: $\{(I \otimes X_{< k})Z^{<k>}\Xt_{> k}\}_{k=1}^d$, $\{(I \otimes V_{< k})Z^{<k>}X_{> k}\}_{k=1}^d$, and $\{(I \otimes X_{< k})Z^{<k>}V_{> k}\}_{k=1}^d$.}

    \textbf{Input: } TT-cores $\{U_1, \ldots U_d\}$, $\{\widetilde{U}_1, \ldots \widetilde{U}_d\}$, variational cores $\{\delta V_1, \ldots, \delta V_d\}$, Euclidean gradient $Z$, observation set $\Omega$ \;

     \For{$k = 1,\ldots, d$}{
        $A_k \leftarrow \text{zeros}(r_{k-1}, n_k, r_k)$ \textit{(cores for computing $(I \otimes X_{<})Z^{<k>}\Xt_{>}$)}\;

        $B_k \leftarrow \text{zeros}(r_{k-1}, n_k, r_k)$ \textit{(cores for computing $(I \otimes V_{<})Z^{<k>}\Xt_{>}$)}\;

        $C_k \leftarrow \text{zeros}(r_{k-1}, n_k, r_k)$ \textit{(cores for computing $(I \otimes X_{<})Z^{<k>}V_{>}$)}\;

     }

     \For{$(i_1, \ldots, i_d) \in \Omega$}{
         (  Precompute left matrix products for $X_<$ and $V_<$)\;

         $U_L\{1\} \leftarrow U_1(i_1)$ \textit{($U_L\{ \}$ and $\delta V_L\{ \}$ are helper variables seperate from input cores)}\; 

         $\delta V_L\{1\} \leftarrow \delta V_1(i_1)$

         \For{$k = 2, \ldots, d-1$}{
             $\delta V_L\{k\} \leftarrow \delta V_L\{k-1\} U_k(i_k) + U_L\{k-1\}\delta V_k(i_k)$\;

             $U_L\{k\} \leftarrow U_L\{k-1\} U_k(i_k)$\;
         }

         \textit{($\Ut_R$, $U_R$, $\delta V_R$ are helper variables)}\;

         $\Ut_R \leftarrow \Ut_d(i_d)$ \;

         $U_R \leftarrow U_d(i_d)$ \;

         $\delta V_R \leftarrow \delta V_d(i_d)$ \;
         
         \textit{(Calculate the cores beginning from the right)}\;

         $A_d(i_d) \leftarrow A_d(i_d) + Z(i_1, \ldots, i_d) \cdot U_L\{d-1\}^\top$\;
         
         $B_d(i_d) \leftarrow B_d(i_d) + Z(i_1, \ldots, i_d) \cdot \delta V_L\{d-1\}^\top$\;
         
         $C_d(i_d) \leftarrow C_d(i_d) + Z(i_1, \ldots, i_d) \cdot U_L\{d-1\}^\top$\;

         \For{$k = d-1,\ldots, 2$} {
            $A_k(i_k) \leftarrow A_k(i_k) + Z(i_1, \ldots, i_d) \cdot U_L\{k-1\}^\top \Ut_R^\top$\;

            $B_k(i_k) \leftarrow B_k(i_k) + Z(i_1, \ldots, i_d) \cdot \delta V_L\{k-1\}^\top \Ut_R^\top$\;

            $C_k(i_k) \leftarrow C_k(i_k) + Z(i_1, \ldots, i_d) \cdot U_L\{k-1\}^\top \delta V_R^\top$\;

            $\Ut_R \leftarrow \Ut_k(i_k) \Ut_R^\top$ \;

            $\delta V_R \leftarrow U_k(i_k) \delta V_R^\top + \delta V_k(i_k) U_R^\top$ \;

            $U_R \leftarrow U_k(i_k) U_R^\top$ \;
         }

         $A_1(i_1) \leftarrow A_1(i_1) + Z(i_1, \ldots, i_d) \cdot \Ut_R^\top$\;

        \textit{($V_{< 1} = 0$, so no update for $B$ here)}\;

         $C_1(i_1) \leftarrow C_1(i_1) + Z(i_1, \ldots, i_d) \cdot U_L\{k-1\}^\top \delta V_R^\top$\;

     }

     \For{$k = 1,\ldots , d$} {
         $(I \otimes X_{< k})Z^{<k>}\Xt_{> k} \leftarrow A_k^L$ \;

         $(I \otimes V_{< k})Z^{<k>}\Xt_{> k} \leftarrow B_k^L$ \;

         $(I \otimes X_{< k})Z^{<k>}V_{> k} \leftarrow C_k^L$ \;
     }
     \caption{Computation of \eqref{comp_essential_1}\eqref{comp_essential_2}\eqref{comp_essential_3} for Tensor Completion (adapted from \citep[Alg. 5.2]{steinlechner2016riemannian})}
\end{algorithm}

\subsection{Computational complexity for the cross terms}
After establishing Corollary \ref{computation_init}, proving the computational complexity for the cross-terms becomes very straight-forward. A fundamental object for the cross terms are the interface matrices for $\P_X^j(Z)$, denoted by $Y^j_{< i}$ and $Y^j_{> i}$. Note that the variational core for $\P_X^j(Z)$ is computable from \eqref{comp_essential_1} by left multiplication of $(I - U_j^L (U_j^L)^\top)$, so, given \eqref{comp_essential_1} has been computed, we can compute all $Y^j_{< i}, Y^j_{> i}$ in $O(dnr^3)$ flops.

For the $j > i, i < d$ case, we can fix $j$ and compute $(Y^j_{< i})^\top \Xt_{>i}$ for all $i < d$ in $O(dnr^3)$ flops by Corollary \ref{computation_init}.\ref{comp_init_2}. We then compute all such matrix products in a total of $O(d^2nr^3)$ flops, and finally the small matrix multiplication of $\delta V_i^L$ is done in $O(nr^3)$ flops each, leading to total complexity of $O(d^2nr^3)$ flops.

The $j > i, i < d$ case follows similarly, where we compute all $(I \otimes V_{< i}^\top)Y^j_{\le i}$ for fixed $j$ in $O(dnr^3)$ flops using an algorithm akin to Algorithm \ref{alg1}. This leads to a total $O(d^2nr^3)$ flops, and finally the remaining $(I - U_k^L(U_k^L)^\top)$ multiplications is done in $O(d^2nr^3)$ flops, leading to total complexity of $O(d^2nr^3)$ flops.

The final case is merely a computation of $(I \otimes V_{< d}^\top)Y^j_{\le d}$, which we have already established can be done in $O(dnr^3)$ flops, leading to total complexity of the entire computation in $O(d^2nr^3)$ flops.

\end{document}